\documentclass[12pt]{amsart}
\usepackage{amsfonts,amsmath,amsthm,amssymb,mathrsfs}
\usepackage[usenames,dvipsnames,svgnames,table]{xcolor}
\usepackage{dynkin-diagrams}
\usepackage{fullpage}
\usepackage{amsaddr}
\usepackage{changes}

\numberwithin{equation}{section}

\newtheorem{thm}{Theorem}
\newtheorem*{thm*}{Theorem}

\newtheorem*{prop*}{Proposition}
\newtheorem{question}[thm]{Question}

\newtheorem{lemma}[thm]{Lemma}

\newtheorem*{conj*}{Conjecture}
\theoremstyle{definition}

\newtheorem{example}[thm]{Example}

\newtheorem*{remark*}{Remark}
\newtheorem*{remarks*}{Remarks}

\newcommand{\Aut}{\operatorname{Aut}}

\newcommand{\Der}{\operatorname{Der}}
\newcommand{\ad}{\operatorname{ad}}
\newcommand{\Ad}{\operatorname{Ad}}

\newcommand{\mf}[1]{\mathfrak{#1}}
\newcommand{\bb}[1]{\mathbb{#1}}
\newcommand{\wt}[1]{\widetilde{#1}}

\DeclareMathOperator{\Isom}{Isom}
\DeclareMathOperator{\Ker}{Ker}

\begin{document}

\title{semi-simple Lie algebras are determined by their Iwasawa subalgebras}
\author{Jonathan Epstein and  Michael Jablonski}

\date{January 18, 2024}
\maketitle


\begin{abstract}  
Using tools from the geometry of Einstein solvmanifolds, we give a geometric argument that a semi-simple Lie algebra (of non-compact type) is completely determined by its Iwasawa subalgebra.  Furthermore, we produce an algebraic procedure for recovering the semi-simple (of non-compact type) from its Iwasawa subalgebra. 
\end{abstract}

Given a (real) semi-simple Lie algebra $\mathfrak g$, one may consider an Iwasawa decomposition $\mathfrak g = \mathfrak k + \mathfrak a + \mathfrak n$.  The  subalgebra  $\mathfrak s = \mathfrak a + \mathfrak n$ of $\mathfrak g$ is called an Iwasawa subalgebra and depends on  1) a Cartan decomposition and 2) a choice of positive roots of the $\ad \mathfrak a$ action on $\mathfrak g$.  Any two Iwasawa subalgebras are conjugate.  To what extent does this subalgebra determine the whole semi-simple Lie algebra?

\begin{thm}\label{thm: iwasawa determines its semi-simple}
Consider two, real semi-simple Lie algebras $\mathfrak g_1$ and $\mathfrak g_2$ of non-compact type with corresponding Iwasawa subalgebras $\mathfrak s_1$ and $\mathfrak s_2$.  If $\mathfrak s_1$ and $\mathfrak s_2$ are isomorphic, then $\mathfrak g_1$ is isomorphic to $\mathfrak g_2$. Moreover, every isomorphism between $\mf{s}_1$ and  $\mf{s}_2$ is the restriction of an isomorphism between $\mf{g}_1$ and $\mf{g}_2$.
\end{thm}

Recall, every semi-simple Lie algebra can be written as a Lie algebra direct sum of (commuting) simple Lie algebras.  If none of these simple factors is compact, then the semi-simple Lie algebra is said to be of non-compact type.

One could deduce the first part of the above theorem from the classification of real semi-simple Lie algebras and symmetric spaces. However, we are not aware of this result in the literature and, besides, this is not a desirable proof. Instead, we exploit well-known results on the geometry of solvmanifolds to avoid using the classification of symmetric spaces.  The second part of the theorem does require a new tool, namely, Lemma \ref{lem: compact derivations of iwasawa}.

If the semi-simple algebra at hand were complex, then one could prove the above theorem without appealing to a classification of  simple Lie algebras by reconstructing the semi-simple Lie algebra from the Borel subalgebra.  However, this approach only works in the real setting when the real semi-simple is split.  In the non-split case, the abelian subalgebra $\mathfrak a$ is not a maximal abelian subalgebra of $\mathfrak g$.

Given that the Iwasawa subalgebra completely determines the semi-simple (in the setting of non-compact type), it is natural to ask if one can rebuild the semi-simple using data from the Iwasawa subalgebra.  Indeed, this is the case.  Recall, given a maximal, fully non-compact, abelian subalgebra $\mathfrak a$ of $\mathfrak g$, we may decompose $\mathfrak g$ as a direct sum of restricted root spaces for the action of $\ad \mathfrak a$ on $\mathfrak g$, i.e.
    \[\mathfrak g = \mathfrak g_0 \oplus \bigoplus_{\lambda\in\Sigma}\mathfrak g_\lambda,\]
where $\Sigma$ is the set of restricted real roots.  The centralizer $\mathfrak g_0$ of $\mathfrak a$ decomposes as $\mathfrak g_0 = \mathfrak m \oplus \mathfrak a$, where $\mathfrak m$ is a compact subalgebra.  Choosing an ordering of our restricted roots, we have a notion of positivity and the Iwasawa subalegbra is 
    \[\mathfrak s = \mathfrak a \oplus \bigoplus_{\lambda \in \Sigma ^+} \mathfrak g_\lambda.\]
In order to rebuild $\mathfrak g$ from $\mathfrak s$, the primary technical challenge is recovering $\mathfrak m$.   
We are unaware of the following result in the literature.

\begin{lemma}\label{lem: compact derivations of iwasawa}  Consider an Iwasawa subalgebra $\mathfrak s$ of $\mathfrak g$, a semi-simple Lie algebra of non-compact type.  Every maximal compact subalgebra of $\Der (\mathfrak s)$ is of the form $\ad_\mathfrak g \mathfrak m$, where $\mathfrak m = Z_\mathfrak k (\mathfrak a)$ for some maximal compact subalgebra $\mathfrak k$ of $\mathfrak g$.
\end{lemma}

After constructing $\mathfrak g_0$ using $\mathfrak m$, one can continue to rebuild the rest of $\mathfrak g$.  These details can be found in Section \ref{sec: construction}.

\begin{remark*}  Interestingly, one cannot extend  Lemma \ref{lem: compact derivations of iwasawa} on compact derivations to (non-compact) reductive derivations.  For example, the Iwasawa subalgebra of $\mathfrak{so}(n,1)$ is $\mathbb R\ltimes \mathbb R^n$, where $\mathbb R$ acts by multiples of the identity.  Since $\mathfrak{gl}(n,\mathbb R)$ commutes with multiples of the identity on $\mathbb R^n$, it can be realized as a subalgebra of derivations of $\mathbb R\ltimes \mathbb R^n$.
\end{remark*}

Finally, one might investigate to what extent the above results extend to the nilradical of the Iwasawa.  Here the real and complex settings diverge.  We address this question in Section \ref{sec: nilradical}.

\section{Proof of Theorem \ref{thm: iwasawa determines its semi-simple}}

For each $i=1,2$, let $G_i$ be the adjoint group with Lie algebra $\mathfrak g_i$.  Let $S_i$ be the subgroup of $G_i$ with Lie algebra $\mathfrak s_i$.  Let $K_i$ be the subgroup of $G_i$ with Lie algebra $\mathfrak k_i$.  Recall, $K_i$ is a maximal compact subgroup of $G_i$ and consider the homogeneous spaces $G_i/K_i$.  Being the adjoint group, $G_i$ has no center and so  acts effectively on $G_i/K_i$.

As is well-known,  each of the   homogeneous spaces $G_i/K_i$ admits a   $G_i$-invariant Einstein metric  $g_i$.  These homogeneous spaces are, in fact, symmetric spaces and so we have  
	$$\Isom_0(G_i/K_i, g_i) = G_i,$$
see Theorem V.4.1 of \cite{Helgason}.  The theory of symmetric spaces is quite robust and actually more than we need.  For our purposes, it is enough to know that the $G_i$ are semi-simple of non-compact type, then we could employ \cite[Theorem 4.1]{Gordon:RiemannianIsometryGroupsContainingTransitiveReductiveSubgroups}.

As $S_1$ acts simply transitively on $G_1/K_1$, the metric $g_1$ induces a left-invariant metric on $S_1$.  Likewise, we obtain a left-invariant Einstein metric  on $S_2$.

Any isomorphism $\phi: \mathfrak s_1 \to\mathfrak s_2$ lifts to an isomorphism
	$$\phi: S_1 \to S_2,$$
as the $S_i$ are simply-connected; we may   pull-back the left-invariant Einstein metric $g_2$ on $S_2$ to a left-invariant Einstein metric $\phi^*g_2$ on $S_1$.  However, left-invariant Einstein metrics on a given solvmanifold are unique up to scaling and automorphism \cite[Theorem 5.1]{Heber}, see also \cite{Lauret:EinsteinSolvmanifoldsAreStandard};   i.e. there exists an automorphism $\phi'$ of $S_1$ and $c>0$ such that
	$$\phi'{ }^*(\phi ^* g_2) = c  g_1.$$
As $\phi'{ }^*(\phi ^* g_2)  = (\phi\circ\phi')^*g_2$ and $\phi \circ\phi' : S_1 \to S_2$ is an isomorphism, we may simplify and replace $ \phi \circ\phi'$ with $\phi$.  Now the isometry $\phi$ between $cg_1$ and $g_2$ yields an isomorphism between their isometry groups.  Thence, we have  an isomorphism between the connected components of the identity for these groups.  These components are precisely $G_1$ and $G_2$ and the induced isomorphism between $\mathfrak g_1$ and $\mathfrak g_2$  gives the desired result.

The proof of the last statement of Theorem \ref{thm: iwasawa determines its semi-simple} is postponed to after the proof of Lemma \ref{lem: compact derivations of iwasawa} in Section \ref{sec: construction}.

\section{Construction of semi-simple of non-compact type from the Iwasawa}\label{sec: construction}
Consider a non-compact, simple Lie algebra $\mathfrak g$ with maximal split torus $\mathfrak a$.  Denoting the restricted real roots of $\mathfrak a$ by $\Sigma$, we have
    \[\mathfrak g = \mathfrak g_0 \oplus \bigoplus_{\lambda \in \Sigma} \mathfrak g_\lambda \]
where $\mathfrak g_0 = \mathfrak m \oplus \mathfrak a$ is the centralizer of $\mathfrak a$ in $\mathfrak g$.  The subalgebra $\mathfrak m$ is compact.  The Iwasawa subalgebra is $\mathfrak s = \mathfrak a \oplus \bigoplus_{\lambda \in \Sigma^+} \mathfrak g_\lambda$, for some choice of positive restricted roots $\Sigma^+$, and we  write $\mathfrak g_{\geq 0} = \mathfrak m \oplus \mathfrak s$, where $\mathfrak s$ is normalized by $\mathfrak m$.  For details, the reader may consult a standard text such as \cite{Onischik-Vinberg:LieGroupsAndAlgebraicGroups} or \cite{Helgason}.

First we remind the reader how one may recover $\mathfrak g$ from $\mathfrak g_{\geq 0}$.  Then we explain how to find $\mathfrak m$ from $\mathfrak s$; this second piece is our new contribution.

\subsection{Building $\mathfrak g$ from $\mathfrak g_{\geq 0}$.}\label{subsec: building g from g>=g0}
To a given real semi-simple Lie algebra $\mathfrak g$, one 
may consider the associated Satake diagram.  The Satake diagram is a decorated version of the Dykin diagram of the complexification $\mathfrak g(\mathbb C)$ of $\mathfrak g$ and it completely determines the real semi-simple Lie algebra at hand. For example, Figure \ref{fig:sample Satake} shows the Dynkin diagram for $E_6$, with decorations. Each vertex is colored black or white and one pair of vertices is connected with an arrow to generate the Satake diagram for the real form EIII.
\begin{figure}[!h]
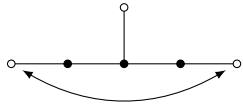

\centering
\dynkin[edge length=.75cm,
involutions={16}]{E}{oo***o}
\caption{The Satake diagram of EIII.}
\label{fig:sample Satake}
\end{figure}
Details for building the Satake diagram of $\mathfrak g$ can be found in \cite[Chapter 5, Section 3]{Onischik-Vinberg:LieGroupsAndAlgebraicGroups}.  We use all the information, and notation, in that work to explain how to recover the Satake diagram of $\mathfrak g$ from $\mathfrak g_{\geq 0} = \mathfrak m \oplus\mathfrak s$.

Let $\mathfrak t$ be a maximal torus of $\mathfrak m$.  The abelian subalgebra $\mathfrak t\oplus\mathfrak a$ complexifies to a Cartan subalgebra of $\mathfrak g(\mathbb C)$. Denoting the roots of $\mathfrak g(\mathbb C)$ by $\Delta$, one has a projection map 
    \[\rho: \Delta \to \Sigma\]
which arises by restricting the linear functions to $\mathfrak a$. Choose an ordering for the roots $\Delta$ in a way that extends positivity from $\Sigma$, i.e.  $\rho(\alpha)>0$ implies $\alpha >0$, for $\alpha\in\Delta$.  Now, for each $\lambda\in\Sigma$, we have 
    \[\mathfrak g_\lambda (\mathbb C) = \bigoplus_{ \{\alpha\in\Delta \ | \ \rho(\alpha)=\lambda\} } \mathfrak g_\alpha.\]
Let $\mf{m}(\mathbb C)^+$ be the sum of the positive root spaces of $\mf{m}(\mathbb C)$. Now we have 
    \[ \mathfrak m(\mathbb C)^+ \oplus \bigoplus_{\lambda > 0} \mathfrak g_\lambda (\mathbb C)\]
is the sum of all the positive root spaces for $\mathfrak g(\mathbb C)$.  Even further, together with the Cartan subalgebra $\mathfrak h = \mathfrak t(\mathbb C)\oplus\mathfrak a(\mathbb C)$, one has the full Borel subalgebra $\mathfrak b$ of $\mathfrak g(\mathbb C)$.  Notice, this subalgebra of $\mathfrak g(\mathbb C)$ is recoverable from $\mathfrak g_{\geq 0}$ alone.

The simple roots form the nodes of the Dynkin diagram of $\mathfrak g(\mathbb C)$.  To recover the edges of the Dynkin diagram, one needs the data contained in the Cartan matrix; this data is determined from the Killing form $B_{\mathfrak g(\mathbb C)}$ of $\mathfrak g(\mathbb C)$ restricted to the Cartan subalgebra $\mathfrak h$.  Denoting the Killing form of the Borel subalgebra by $B_\mathfrak b$, we have the following immediate relation
    \[B_{\mathfrak g(\mathbb C)} | _{\mathfrak h \times \mathfrak h} = \frac 1 2 B_{\mathfrak b} | _{\mathfrak h \times \mathfrak h}.\]
In this way, we can build the Dynkin diagram of $\mathfrak g(\mathbb C)$ using data from $\mathfrak g_{\geq 0}$.

To construct the Satake diagram of $\mathfrak g$, we need to decorate the Dynkin diagram of $\mathfrak g(\mathbb C)$ by coloring the nodes and adding arrows, where appropriate.

Denote the simple roots of $\mathfrak g(\mathbb C)$ above by $\Pi$.  We may write $\Pi = \Pi_1\cup \Pi_0$ where $\Pi_0=\Pi\cap \Ker \rho$.  (The projection $\rho(\Pi_1)$ turns out to be a set of simple roots for $\Sigma$.) 
The nodes corresponding to $\Pi_0$ are colored black; the nodes corresponding to $\Pi_1$ are colored white. These are the compact and non-compact simple roots, respecitvely.

Assigning possible arrows is done by studying each restricted root space $\mathfrak g_\lambda$, for $\lambda\in\rho(\Pi_1)$, to see if $\mathfrak g_\lambda(\mathbb C)$ decomposes into one or two root spaces under the action of $\mathfrak t\oplus \mathfrak a$.  If we have $\alpha_1,\alpha_2\in\Pi_1$ such that $\rho(\alpha_1)=\rho(\alpha_2)\in\Sigma$, then either $\alpha_1=\alpha_2$ or these are the only two such roots and we connect those nodes of the Dynkin diagram with an arrow to build the Satake diagram.

This completes the construction of the Satake diagram of $\mathfrak g$ using data coming from $\mathfrak g_{\geq 0} = \mathfrak m \oplus \mathfrak s$.

\subsection{Recovering $\mathfrak m$ from $\mathfrak s$.} \label{sec: recovering m from s}
To complete the argument that $\mathfrak g$ can be reconstructed from $\mathfrak s$, we need to recover $\mathfrak m$ from $\mathfrak s$. We now prove  Lemma \ref{lem: compact derivations of iwasawa}, which shows that this compact subalgebra can be realized as a maximal compact subalgebra of $\Der(\mathfrak s)$.  We start with the following lemma which is also used in \cite{GordonJablonski:EinsteinSolvmanifoldsHaveMaximalSymmetry}.

\begin{lemma}\label{lem: Compact aut in isometry of Einstein}  Let $S$ be a completely solvable Lie group admitting a left-invariant Einstein metric.  Let $H$ be a compact subgroup of $Aut(S)$. There exists some left-invariant Einstein metric $g$  such that $H\subset \Isom(S,g)$.
\end{lemma}

\begin{proof}
This follows from the proof of Theorem 4.1 in \cite{Jablo:ConceringExistenceOfEinstein}.  There the solvable groups of interest are unimodular and the metrics are solvsolitons, but Einstein metrics are a special case of solvsolitons and  the first part of the proof there, which is all that is needed,  applies to all completely solvable groups, not just unimodular. 
\end{proof}

In the case of interest, i.e.\ when $S$ is an Iwasawa group, the Einstein metric from the lemma will be a symmetric metric.  It is quick to see that the isometry group does not change if one varies which symmetric metric is being used, even when one has a product of simple groups and the metric is not necessarily Einstein.

\begin{proof}[Proof of Lemma \ref{lem: compact derivations of iwasawa}]
Let $\mathfrak h$  be a maximal compact subalgebra $\mathfrak h$ of $\Der(\mathfrak s)$ and let $H$ be the corresponding compact group of $\Aut(\mathfrak s)$; recall, $S$ being simply-connected implies $\Aut(\mathfrak s) = \Aut(S)$.

Let $G$ be the adjoint group of $\mathfrak g$, then we have $G = \Isom_0(S,g)$ for a symmetric metric, as above.  Thus $S=G/K$ where $K$ is any maximal compact subgroup of $G$.  Choosing $K$ so that $e\in S$ corresponds to $eK \in G/K$, we have that the isotropy of $G$ at $e\in S$ is $\Ad\,K$.  Since  $H$ fixes $e\in S$, Lemma \ref{lem: Compact aut in isometry of Einstein} gives that $H$ is in the isotropy at this point and we have $H\subset \Ad\,K$.

So far, we have obtained 
	$$\mathfrak h\subset \ad\,\mathfrak k\subset \ad\,\mathfrak g,$$
where $\ad = \ad_\mathfrak g$ is the adjoint action relative to $\mathfrak g$.  
As $\mathfrak h$ normalizes $\mathfrak s$, it fixes some maximal reductive subalgebra $\mathfrak a$ of $\mathfrak s$, i.e.~$\mathfrak h \subset \ad\,N_\mathfrak g(\mathfrak a)$.  Recall, $\mathfrak s = \mathfrak a \oplus \mathfrak n$, where $\mathfrak n$ is the nilradical of $\mathfrak s$.    
It is well-known that there is some possibly different maximal compact $\mathfrak k$ of $\mathfrak g$ such that
	$$N_\mathfrak g (\mathfrak a) = \mathfrak a \oplus \mathfrak m,$$
where $\mathfrak m = Z_\mathfrak k (\mathfrak a)$, see Chapter IX of \cite{Helgason}.  As $\mathfrak h$ is a compact subalgebra and $\mathfrak a \oplus \mathfrak m$ is abelian with a unique maximal compact subalgebra $\mathfrak m$, we see that $\mathfrak h \subset \ad\,\mathfrak m$.  

The work above proves the lemma in the event that $\mathfrak h$ fixes a predetermined choice of maximal reductive $\mathfrak a$ of $\mathfrak s$.  However,  since any compact subalgebra of $\Der(\mathfrak s)$ always fixes some maximal reductive $\mathfrak a$ and the choice of $\mathfrak a$ is unique up to conjugation \cite{Mostow:FullyReducibleSubgrpsOfAlgGrps},   the claim follows.

\end{proof}

\begin{proof}[Proof of Theorem \ref{thm: iwasawa determines its semi-simple}, cont.] Let $\mf{s}_i \subset \mf{g}_i$, $i = 1,2$, be two Iwasawa subalgebras, and let $\varphi : \mf{s}_1 \to \mf{s}_2$ be an isomorphism. The isomorphism $\varphi$ induces an isomorphism $\varphi': \Der(\mf{s}_1) \to \Der(\mf{s}_2)$. Fix a maximal compact subalgebra of $\Der(\mf{s}_1)$ and identify it with $\mf{m}_1 \subset \mf{g}_1$ as in Lemma \ref{lem: compact derivations of iwasawa}.  Likewise, we identify $\varphi'(\mf m_1)$ with $\mf m_2\subset \mf g_2$. 
Using these identifications and the induced isomorphism $\varphi'$ of derivation algebras, we obtain an isomorphism $\mf{m}_1 \ltimes \mf{s}_1 \to \mf{m}_2 \ltimes \mf{s}_2$ which restricts to $\varphi$ on $\mf{s}_1$. Abusing notation, this isomorphism between the $\mathfrak m_i\ltimes \mathfrak s_i$ is also denoted by $\varphi$.

Let $\mf{t}_1 \subset \mf{m}_1$ be a maximal abelian subalgebra and consider the maximal abelian   $\mf{t}_2 = \varphi(\mf{t}_1)$ of $\mf m_2$.  The subalegbra $\mf{h}_i = \mf{t}_i \oplus \mf{a}_i$  is a Cartan subalgebra of $\mf{g}_i$, for each $i=1,2$, with $\mf h_2=\varphi (\mf h_1)$.  In this way,  the complexification
\begin{align*}
    \varphi(\bb{C}) : \mf{m}_1(\bb{C}) \ltimes \mf{s}_1(\bb{C}) \to \mf{m}_2(\bb{C}) \ltimes \mf{s}_2(\bb{C})
\end{align*}
carries the Cartan subalgebra $\mf{h}_1(\bb{C})$ isomorphically onto the Cartan subalgebra $\mf{h}_2(\bb{C})$. Moreover, since $\varphi(\bb{C})$ preserves brackets, its transpose carries the root system $\Delta_2 \subset \mf{h}_2(\bb{C})^*$  bijectively onto the root system of $\Delta_1 \subset \mf{h}_1(\bb{C})^*$. 

As in Section \ref{sec: construction}, $\mf{m}_i(\bb{C}) \ltimes \mf{s}_i(\bb{C})$ contains all the simple root spaces of $\mf{g}_i(\bb{C})$. Let $\{ E_1, \ldots, E_\ell \}$ be a set of simple root vectors of $\mf g_1(\bb C)$. By the Isomorphism Theorem (cf \cite{Knapp:LieGroupsBeyondAnIntroduction}, Theorem 2.108), there exists a unique Lie algebra isomorphism $\wt{\varphi}: \mf{g}_1(\bb{C}) \to \mf{g}_2(\bb{C})$ such that $\wt{\varphi}$ agrees with $\varphi(\bb{C})$ on $\mf{h}_1(\bb{C})$ and $\wt{\varphi}(E_j) = \varphi(\bb{C})(E_j)$. 

Next, we argue that $\wt \varphi$ agrees with $\varphi$ on $\mf s_1$.  Recall, $\mf s_1 = \mf a_1\oplus \mf n_1$, as a vector space.  By construction,  $\wt \varphi$ agrees with $\varphi(\mathbb C)$ on the subspace $\mf a_1\subset \mf h_1 (\bb C)$; furthermore, this is simply $\varphi$ on $\mf a_1$.  To see that $\wt \varphi$ and $\varphi$ agree on $\mf n_1$, it suffices to verify this on the simple restricted roots spaces as these are contained in $\mf n_1$, generate $\mf n_1$ as a Lie algebra, and $\wt \varphi$ and $\varphi$ are homomorphisms.

Let $X\in\mf n_1$ be a simple restricted root vector.  From the discussion in Section \ref{subsec: building g from g>=g0}, we have that $X=z_1E_1+\dots z_\ell E_\ell$ for some $z_i\in\bb C$.  Using $\bb C$-linearity of $\wt \varphi$ and $\varphi(\bb C)$, the definition of $\wt \varphi$, and that $\varphi(\bb C)$ agress with $\varphi$ on $\mathfrak s$, we have the following.
\begin{align*}
    \wt{\varphi}(X) &= z_1 \wt{\varphi}(E_1) + \cdots + z_\ell \wt{\varphi}(E_\ell) \\
    &= z_1 \varphi(\bb{C})(E_1) + \cdots + z_\ell \varphi(\bb{C})(E_\ell) \\
    &= \varphi(\bb{C}) \left( z_1 E_1 + \cdots + z_\ell E_\ell \right) \\
    &= \varphi(\bb{C}) \left( X \right) \\
    &= \varphi(X)
\end{align*}
This completes the proof that $\wt\varphi$ agress with $\varphi$ on $\mathfrak s_1$.

Lastly, we show that $\wt{\varphi}$ restricts to a real Lie algebra isomorphism of the $\mf{g}_i$. Let $\sigma_i:\mf{g}_i(\bb{C}) \to \mf{g}_i(\bb{C})$ be the anti-involution whose fixed set is $\mf{g}_i$. Then
\begin{align*}
    \sigma_2 \circ \wt{\varphi} \circ \sigma_1 : \mathfrak g_1(\mathbb C) \to \mathfrak g_2(\mathbb C)
\end{align*}
is a $\bb{C}$-linear isomorphism that agrees with $\wt{\varphi}$ on $\mf{m}_1(\bb{C}) \oplus \mf{a}_1(\bb{C})$ and on the simple root spaces.
By the uniqueness part of the Isomorphism Theorem, we must have $\sigma_2 \circ \wt{\varphi} \circ \sigma_1 = \wt{\varphi}$.  Then
\begin{align*}
    \wt{\varphi}(\mf{g}_1) = (\sigma_2\circ \wt{\varphi} \circ\sigma_1)(\mf{g}_1) = (\sigma_2\circ \wt{\varphi}) (\mf{g}_1) = \sigma_2(\wt{\varphi}(\mf{g}_1))
\end{align*}
shows that $\sigma_2$ fixes $\wt{\varphi}(\mf{g}_1)$ and hence $\wt{\varphi}(\mf{g}_1) = \mf{g}_2$. In this way, the isomorphism $\varphi:\mf s_1\to\mf s_2$ is the restriction of the isomorphism $\wt \varphi \big|_{\mf g_1} :\mf g_1\to \mf g_2$.

\end{proof}

\section{Nilradical}\label{sec: nilradical}
If $\mathfrak g$ is a complex semi-simple Lie algebra, one can recover the Borel subalgebra $\mathfrak b$ from its nilradical $\mathfrak n$, and hence all of $\mathfrak g$ from this nilpotent subalgebra.  To rebuild $\mathfrak b$ from $\mathfrak n$, one needs a Cartan subalgebra $\mathfrak h$ of $\mathfrak g$.

We will show how the Cartan subalgebra can be chosen to be a maximal abelian, reductive subalgebra of $\Der(\mathfrak n)$.  Fix a Cartan subalgebra $\mathfrak h$ of $\mathfrak g$.  As this subalgebra acts faithfully on $\mathfrak n$, we may view it as a subalgebra of some maximal abelian, reductive subalgebra of $\Der(\mathfrak n)$, denoted by $\mathfrak l$.  By choice, we have $\mathfrak h \subset \mathfrak l$.

A derivation of $\mathfrak n$ is completely determined by its values on any complement of $[\mathfrak n,\mathfrak n]$ in $\mathfrak n$.  In terms of the roots of $\mathfrak g$, we may choose the complement of $[\mathfrak n,\mathfrak n]$ to be $\bigoplus_{\alpha\in\Pi}\mathfrak g_\alpha$, where $\Pi$ is the set of simple roots.  Now, as these roots spaces are all 1-dimensional, we have that $\mathfrak l \subset \mathfrak{gl} (\mathfrak g_{\alpha_1}) \times \dots \times \mathfrak{gl} (\mathfrak g_{\alpha_k})$. But $\dim \mathfrak h = |\Pi| = k$ and so $ \mathfrak{gl} (\mathfrak g_{\alpha_1}) \times \dots \times \mathfrak{gl} (\mathfrak g_{\alpha_k}) \simeq \mathfrak h$.  Thus, $\mathfrak h = \mathfrak l$, as claimed.

In the same way, the Iwasawa of a a split real semi-simple Lie algebra can be recovered from its nilradical.  However, this is not true, in general, for real semi-simple Lie algebras.

\begin{example}
    The 3-dimensional Heisenberg Lie algebra is the nilradical of the Iwasawa subalgebra for both $\mathfrak{sl}(3,\mathbb R)$ and $\mathfrak{su}(2,1)$, which are non-isomorphic simple, real Lie algebras.
\end{example}

Given an Iwasawa subalgebra $\mathfrak s=\mathfrak a\oplus \mathfrak n$, of some non-compact simple Lie algebra, we recall that $\mathfrak a$ contains a special element $\varphi$, called the pre-Einstein derivation, such that the Lie group $S_0$, with Lie algebra $\mathfrak s_0=\mathbb R(\varphi)\ltimes \mathfrak n$, admits a left-invariant Einstein metric, cf. \cite{Nikolayevsky:EinsteinSolvmanifoldsandPreEinsteinDerivation}; see also \cite{Heber,Lauret:EinsteinSolvmanifoldsAreStandard}.  At the other extreme, one may consider a maximal fully non-compact, reductive subalgebra $\overline{\mathfrak a}$ of $\Der(\mathfrak n)$ containing $\mathfrak a$.  The Lie group $\overline S$, with Lie algebra $\overline{\mathfrak s} = \overline{\mathfrak a }\ltimes \mathfrak n$, also admits an Einstein metric, as $\overline{\mathfrak a}$ contains the pre-Einstein derivation $\varphi$.

\begin{question}
    Can one characterize the intermediate subalgebras between $\mathbb R(\varphi)$ and $\overline{\mathfrak a}$ whose corresponding intermediate Lie subgroups between $S_0$ and $\overline S$ are symmetric spaces?
\end{question}

\subsubsection*{Acknowledgments}  We would like to thank Victor Cortes for comments on an early version of this manuscript.  This work was supported in part by National Science Foundation grant DMS-1906351.

\providecommand{\bysame}{\leavevmode\hbox to3em{\hrulefill}\thinspace}
\providecommand{\MR}{\relax\ifhmode\unskip\space\fi MR }
\providecommand{\MRhref}[2]{%
  \href{http://www.ams.org/mathscinet-getitem?mr=#1}{#2}
}
\providecommand{\href}[2]{#2}

\end{document}